\documentclass[12pt,reqno]{article}

\usepackage[usenames]{color}
\usepackage{amssymb}
\usepackage{amsmath}
\usepackage{amsthm}
\usepackage{amsfonts}
\usepackage{amscd}
\usepackage{graphicx}

\usepackage[colorlinks=true,
linkcolor=webgreen,
filecolor=webbrown,
citecolor=webgreen]{hyperref}

\definecolor{webgreen}{rgb}{0,.5,0}
\definecolor{webbrown}{rgb}{.6,0,0}

\usepackage{color}

\usepackage{graphics}
\usepackage{latexsym}

\begin{document}

\theoremstyle{plain}
\newtheorem{theorem}{Theorem}
\newtheorem{corollary}[theorem]{Corollary}
\newtheorem{lemma}[theorem]{Lemma}
\newtheorem{proposition}[theorem]{Proposition}

\theoremstyle{definition}
\newtheorem{definition}[theorem]{Definition}
\newtheorem{example}[theorem]{Example}
\newtheorem{conjecture}[theorem]{Conjecture}

\theoremstyle{remark}
\newtheorem{remark}[theorem]{Remark}

\begin{center}
\vskip 1cm{\LARGE\bf A Note on the Gessel Numbers
\vskip 1cm}
\large
Jovan Miki\'{c}\\
University of Banja Luka\\
Faculty of Technology\\
Bosnia and Herzegovina\\
\href{mailto:jovan.mikic@tf.unibl.org}{\tt jovan.mikic@tf.unibl.org} \\
\end{center}

\vskip .2in

\begin{abstract}
The Gessel number $P(n,r)$ represents the number of lattice paths in a plane with unit horizontal and vertical steps from $(0,0)$ to $(n+r, n+r-1)$ that never touch any of the points from the set $\{(x,x) \in \mathbb{Z}^2: x \geq r\}$. In this paper, we use combinatorial arguments to derive  a recurrence  relation between $P(n,r)$ and $P(n-1,r+1)$. Also, we give a new proof for a well-known closed formula for $P(n,r)$. Moreover, a new combinatorial interpretation for the Gessel numbers is presented.

\end{abstract}

\noindent\emph{ \textbf{Keywords:}} Gessel numbers, Catalan numbers, central binomial coefficient, lattice paths.

\noindent \textbf{2020} {\it \textbf{Mathematics Subject Classification}}:
Primary 05A10; Secondary 11B65.

\section{Introduction}\label{l:1}

Let $n$ be a non-negative integer and let $r$ be a positive integer.
The Gessel number $P(n,r)$  counts all lattice paths  \cite[p.\ 191]{IG} in plane with $(1,0)$ and $(0,1)$ steps from $(0, 0)$ to $(n+r, n+r-1)$ that never touch 
any of the points from the set $\{(x,x) \in \mathbb{Z}^2: x \geq r\}$.

By using a combinatorial argument and an instance of the Pfaff-Saalsch\"{u}tz theorem, Gessel proved that (see \cite[p.\ 191]{IG})
\begin{equation}\label{eq:1}
P(n,r)=\frac{r}{2(n+r)}\binom{2n}{n}\binom{2r}{r}\text{.}
\end{equation}

Let $C_n=\frac{1}{n+1}\binom{2n}{n}$ denote the $n$th Catalan number. The Catalan numbers are well studied in the literature (e.g.\ see \cite{TK, RS1}) and there are many counting problems in combinatorics whose solutions are given by the Catalan numbers. For example, the Catalan number $C_n$ is the number of all paths in a  plane from $(0,0)$ to $(n,n)$ with $(1,0)$ and $(0,1)$ steps such that they never rise above
the line $y=x$ (see \cite[Example 9.1]{TK}, \cite[Problem 158]{RS}, and \cite[Eq.~(10.11)]{CK}).

By setting $r=1$ in Equation~(\ref{eq:1}), it follows that $P(n,1)=C_n$.

Recently, it is shown \cite[Cor.\ 4]{JM1} that the sum $\sum_{k=0}^{2n}(-1)^k\binom{2n}{k}^m C_k C_{2n-k}$ is divisible by $\binom{2n}{n}$ for all non-negative integers $n$ and for all positive integers $m$. The Gessel numbers play an interesing role in this proof \cite[Eq.~(68)]{JM1}.

It is known \cite[p.\ 191]{IG} that, for a fixed positive integer $r$, the smallest positive integer $K_r$ such that $\displaystyle{\frac{K_r}{n+r}\binom{2n}{n}}$ is an integer for every $n$ is $\displaystyle{\frac{r}{2}\binom{2r}{r}}$. Guo gave a generalization of  the Gessel numbers \cite[Eq.~(1.10)]{VG2}.

In this paper, we study the Gessel numbers by establishing a close relationship between the Gessel and  the Catalan numbers. After we establish such a relationship, we present a new combinatorial interpretation for the Gessel Numbers.  This combinatorial interpretation is contained in our main result,  Theorem~\ref{t:1}.

\section{Main results}
Our first result expresses $P(n,r)$ in terms of the Catalan numbers.
 
\begin{proposition}\label{lm:1}
Let $n$ be a non-negative integer and let $r$ be a positive integer. Then
\begin{equation}\label{eq:2}
P(n,r)=\sum_{k=0}^{r-1}\binom{2k}{k}C_{n+r-k-1}\text{.}
\end{equation}
\end{proposition}
\begin{proof}
Let $k$ be a non-negative integer, and let $p$ be a  path from $(0,0)$ to $(n+r,n+r-1)$, with $(1,0)$ and $(0,1)$ steps, that never touches 
any of the points from the set $\{(x,x) \in \mathbb{Z}^2: x \geq r\}$. Assume that the last point point of intersection of $p$ and the line $y=x$, when taking the path $p$ from left to right, is the point $(k,k)$, where $0 \leq k \leq r-1$.

The number of ``permitted" paths from $(0,0)$ to $(k,k)$ is $\binom{2k}{k}$, since every such path does not contain ``forbidden" points. Every such path passes through the point $(k+1,k)$, as $(k,k)$ is the last intersection point between this path and the line $y=x$. Note that the segment that connects points $(k+1,k)$ and $(n+r,n+r-1)$ is parallel
with the line $y=x$. 

The number of ``permitted" paths from $(k+1,k)$ to $(n+r, n+r-1)$  is the same as the number of all paths from  $(k+1,k)$ to $(n+r, n+r-1)$ with $(1, 0)$ and $(0,1)$ steps  that never rise above the line $y=x-1$. It follows that there are $C_{n+r-k-1}$ such paths.

Therefore, the number of all paths whose last  intersection with the line $y=x$ is the point $(k,k)$  is $\binom{2k}{k} C_{n+r-k-1}$. Since $k$ takes values  from $0$ to $r-1$, we obtain Equation~(\ref{eq:2}). 
\end{proof}

We use the previous proposition to  obtain a recurrence relation between $P(n-1,r+1)$  and $P(n,r)$. 
\begin{theorem}\label{Cor:1}
Let $n$ and $r$ be positive integers. Then
\begin{equation}\label{eq:3}
P(n-1,r+1)-P(n,r)=\binom{2r}{r}C_{n-1}.
\end{equation}
\end{theorem}
\begin{proof}
The number $P(n-1,r+1)$ counts all paths from $(0,0)$ to $(n+r,n+r-1)$, with $(1,0)$ and $(0,1)$ steps, that never touch 
any of the points from the set $\{(x,x) \in \mathbb{Z}^2: x \geq r+1\}$. 

Therefore, the number $ P(n-1,r+1)-P(n,r)$ counts all paths from $(0,0)$ to $(n+r,n+r-1)$, with $(1,0)$ and $(0,1)$ steps, whose last point of intersection with the line $y=x$ is the point $(r,r)$. 

The number of ``permitted" paths from $(0,0)$ to $(r,r)$ is $\binom{2r}{r}$, since every such path does not contain ``forbidden" points. 

After $(r,r)$, every such  path must pass through $(r+1,r)$. Note that the segment that connects points $(r+1,r)$ and $(n+r,n+r-1)$ is parallel
with the line $y=x$. The number of ``permitted" paths from $(r+1,r)$ to $(n+r, n+r-1)$  is the same as the number of all paths from  $(r+1,r)$ to $(n+r, n+r-1)$, with $(1, 0)$ and $(0,1)$ steps, that never rise above the line $y=x-1$. By a well-known  \cite[Eq.~(10.11)]{CK}) property of the Catalan numbers, it follows that there are $C_{n-1}$ such paths.

Therefore,  the number of all paths whose last point of intersection with the line $y=x$ is the point $(r,r)$ is $\binom{2r}{r} C_{n-1}$.
\end{proof}


By using recurrence relation (\ref{eq:3}) and induction on $n$, we give a proof of Equation~(\ref{eq:1}). 

Let $S(n,r)$ denote $\displaystyle{\frac{r}{2(n+r)}\binom{2n}{n}\binom{2r}{r}}$. We will show that $P(n,r)=S(n,r)$ for all non-negative integers $n$ and for all positive integers $r$.  We use induction on $n$. 


For $n=0$, since the ``final" point $(r,r-1)$ is below the first ``forbidden" point $(r,r)$, the Gessel number $P(0,r)$ counts all paths in a plane from $(0,0)$ to $(r,r-1)$ with $(1,0)$ and $(0,1)$ steps  without any restrictions. Hence, $\displaystyle{P(0,r)=\binom{2r-1}{r}}$ or $\displaystyle{P(0,r)=\frac{1}{2}\binom{2r}{r}}$.

Therefore, it follows that $P(0,r)=S(0,r)$ for all positive integers $r$. 

Let us assume that $P(n-1,r)=S(n-1,r)$ for some positive integer $n$ and for all positive integers $r$.

We use a well-known \cite[p.\ 26]{TK} recurrence relation for the central binomial coefficient:
\begin{equation}\label{eq:6}
\binom{2(r+1)}{r+1}=\frac{2(2r+1)}{r+1}\binom{2r}{r}\text{.}
\end{equation}

Then we have that the following equalities hold:
\begin{align*}
P(n,r)&=P(n-1,r+1)-\binom{2r}{r}C_{n-1}&&(\text{by  Equation~(\ref{eq:3})})\\
&=S(n-1,r+1)-\binom{2r}{r}C_{n-1}&& (\text{by the induction hypothesis})\\
&=\frac{r+1}{2(n+r)}\binom{2(n-1)}{n-1}\binom{2(r+1)}{r+1}-\binom{2r}{r}C_{n-1}\\
\phantom{aaa}&=\frac{r+1}{2(n+r)}\binom{2(n-1)}{n-1}\frac{2(2r+1)}{r+1}\binom{2r}{r}-\binom{2r}{r}C_{n-1}&&(\text{by Equation~(\ref{eq:6})})\\
\end{align*}
\begin{align*}
&=\frac{2r+1}{n+r}\binom{2(n-1)}{n-1}\binom{2r}{r}-\frac{1}{n}\binom{2r}{r}\binom{2(n-1)}{n-1}\\
&=\binom{2(n-1)}{n-1}\binom{2r}{r}\bigl{(}\frac{2r+1}{n+r}-\frac{1}{n}\bigr{)}\\
&=\binom{2(n-1)}{n-1}\binom{2r}{r}\frac{r(2n-1)}{n(n+r)}\\
&=\frac{r}{2(n+r)}\binom{2r}{r}\bigl{(}\frac{2(2n-1)}{n}\binom{2(n-1)}{n-1}\bigr{)}\\
&=\frac{r}{2(n+r)}\binom{2r}{r}\binom{2n}{n} &&(\text{by Equation~(\ref{eq:6})})\\
&=S(n,r)\text{.}\\
\end{align*}
This completes our proof by induction. \vspace{4mm}

\begin{definition}\label{def:1}
Let $n$ be a non-negative integer and let $r$ be a positive integer. The number $Q(n,r)$ counts all lattice paths in a plane with $(1,0)$ and $(0,1)$ steps from $(0, 0)$ to $(n+r, n+r-1)$ that never touch 
any of the points from the set $\{(x,x) \in \mathbb{Z}^2: 1 \leq x \leq r  \}$.
\end{definition}


\begin{proposition}\label{lm:2}
Let $n$ and $r$ be positive integers. Then
\begin{equation}\label{eq:4}
Q(n,r)=\sum_{k=1}^{n}C_{r+k-1}\binom{2(n-k)}{n-k}\text{.}
\end{equation}
\end{proposition}
\begin{proof}

Let $k$ be a non-negative integer, and let $p$ be a  path from $(0,0)$ to $(n+r,n+r-1)$, with $(1,0)$ and $(0,1)$ steps, that never touches 
any of the points from the set $\{(x,x) \in \mathbb{Z}^2: 1 \leq x \leq r\}$. There are two cases to consider.

\textbf{The first case:}
A path $p$ intersects the line $y=x$ only at the point $(0,0)$. In this case, a path $p$ must begin with a $(1,0)$ step. Note that the segment that connects the points $(1,0)$ and $(n+r,n+r-1)$ is parallel with the line $y=x$.
The number of ``permitted" paths from $(1,0)$ to $(n+r, n+r-1)$  is the same as the number of all paths from  $(1,0)$ to $(n+r, n+r-1)$, with $(1, 0)$ and $(0,1)$ steps, that never rise above the line $y=x-1$. It follows that there are $C_{n+r-1}$ such paths.

\textbf{The second case:}
A path $p$ intersects the line $y=x$ in at least two points. 
Let $(r+k,r+k)$ be the first point of intersection between $p$ and the line $y=x$ after the point $(0,0)$. Here, $1\leq k \leq n-1$. Note that, in this case, $ n\geq 2$.

Let $m$ be a positive integer. It is readily verified that the number of all paths in a plane from $(0,0)$ to $(m,m)$, with $(1,0)$ and $(0,1)$ steps,  that intersect the line $y=x$ only at points $(0,0)$ and $(m,m)$  is $2C_{m-1}$.
Therefore, the number of ``permitted" paths from $(0,0)$ to $(r+k,r+k)$ is $2C_{r+k-1}$.

The number of ``permitted" paths from $(r+k,r+k)$ to $(n+r,n+r-1)$ is the same as the number of all paths from $(r+k,r+k)$ to $(n+r,n+r-1)$ with $(1,0)$ and $(0,1)$ steps, since every such path does not contain ``forbidden" points. 
It follows that there are $\displaystyle{\frac{1}{2}\binom{2(n-k)}{n-k}}$ such paths (see \cite[Equation~(10.3)]{CK}).

Therefore, the number of all paths that intersect the line $y=x$ at the point $(r+k,r+k)$ for the first time after the point $(0,0)$
 is $C_{r+k-1}\binom{2(n-k)}{n-k}$. Since  $k$ can take values from  $1$ to $n-1$, it follows that there are $\displaystyle{\sum_{k=1}^{n-1}C_{r+k-1}\binom{2(n-k)}{n-k}}$ such paths.

Putting ever together, it holds:
\begin{align*}
Q(n,r)&=C_{n+r-1}+\sum_{k=1}^{n-1}C_{r+k-1}\binom{2(n-k)}{n-k}\\
&=\sum_{k=1}^{n}C_{r+k-1}\binom{2(n-k)}{n-k}\text{.}
\end{align*}
\end{proof}

\begin{remark}\label{rem:2}
Note that, for $n=0$, the number $Q(0,r)$ is equal to $C_{r-1}$.
\end{remark}

We now  use Proposition~\ref{lm:1} and Proposition~\ref{lm:2} to prove our main result that gives us a new combinatorial interpretation for the Gessel numbers.

\begin{theorem}\label{t:1}
Let $n$ and $r$ be positive integers. Then 
\begin{equation}\label{eq:5}
P(n,r)=Q(r,n)\text{.}
\end{equation}
\end{theorem}
\begin{proof}
By setting $n:=r$ and $r:=n$ in Proposition~\ref{lm:2}, it follows that
\begin{equation}\label{eq:7}
Q(r,n)=\sum_{k=1}^{r}C_{n+k-1}\binom{2(r-k)}{r-k}\text{.}
\end{equation}

By  substituting  $t$ for $r-k$, it follows that Equation~(\ref{eq:7}) becomes
\begin{equation}\label{eq:8}
Q(r,n)=\sum_{t=0}^{r-1}C_{n+r-t-1}\binom{2t}{t}\text{.}
\end{equation}

By using Proposition~\ref{lm:1} and Equation~(\ref{eq:8}), it follows that $Q(r,n)=P(n,r)$. 
\end{proof}

Theorem \ref{t:1} gives a  new combinatorial interpretation for Gessel numbers. By Theorem \ref{t:1}, the Gessel number $P(n,r)$ is the number of all lattice paths  in a plane with $(1,0)$ and $(0,1)$ steps from $(0, 0)$ to $(n+r, n+r-1)$ that never touch 
any of the points from the set  $\{(x,x) \in \mathbb{Z}^2: 1 \leq x \leq n  \}$.






\section{Concluding remarks}\label{l:6}

We end this paper with some formulas for the Gessel numbers.

Let $n$ be a non-negative integer, and let $r$ be a positive integer.
By Equation~(\ref{eq:1}), Theorem \ref{t:1}, and Remark~\ref{rem:2}, it follows that
\begin{equation}\label{eq:9}
Q(n,r)=\begin{cases}
C_{r-1}\text{, if }n=0\\
\frac{n}{2(n+r)}\binom{2n}{n}\binom{2r}{r}\text{, if } n>0. 
\end{cases}
\end{equation}

Let $n$ and $r$ be positive integers. Then the following formulas are true:
\begin{align}
\frac{1}{2}\binom{2n+2r}{n+r}-P(n,r)&=\sum_{k=1}^{n}\binom{2(r+k-1)}{r+k-1}C_{n-k}\text{,}\label{eq:10}\\
\frac{1}{2}\binom{2n+2r}{n+r}-Q(n,r)&=\sum_{l=1}^{r}\binom{2(n+r-l)}{n+r-l}C_{l-1}\text{.}\label{eq:11}
\end{align}

The left side of Equation~(\ref{eq:10}) represents the number of all lattice paths in a plane with $(1,0)$ and $(0,1)$  steps from $(0,0)$ to $(n+r, n+r-1)$ whose intersection with the set $\{(x,x) \in \mathbb{Z}^2: r\leq x \leq n+r-1\}$ is non-empty. It is readily verified that there are  $\binom{2(r+k-1)}{r+k-1}C_{n-k}$ lattice paths  in a plane with $(1,0)$ and $(0,1)$  steps from $(0,0)$ to $(n+r, n+r-1)$ whose last point of intersection with the set  $\{(x,x) \in \mathbb{Z}^2: r\leq x \leq n+r-1\}$ is the point $(r+k-1,r+k-1)$. Here, $1\leq k\leq n$.

Similarly, the left side of Equation~(\ref{eq:11}) represents the number of all lattice paths in a plane with $(1,0)$ and $(0,1)$  steps from $(0,0)$ to $(n+r, n+r-1)$ whose intersection with the set $\{(x,x) \in \mathbb{Z}^2: 1\leq x \leq r\}$ is non-empty. It is readily verified that there are $\binom{2(n+r-l)}{n+r-l}C_{l-1}$ lattice paths  in a plane with $(1,0)$ and $(0,1)$  steps from $(0,0)$ to $(n+r, n+r-1)$ whose first point of intersection with the set  $\{(x,x) \in \mathbb{Z}^2: 1\leq x \leq r\}$, after the point $(0,0)$, is the point $(l,l)$. Here, $1\leq l\leq r$.

Note that, by using Equations~(\ref{eq:10}) and (\ref{eq:11}), one can give another proof of Theorem \ref{t:1}.

\begin{remark}\label{rem:3}
By using a combinatorial argument, Gessel proved \cite[Equation~(39)]{IG} the following formula:
\begin{equation}\label{eq:12}
\sum_{k=0}^{n}P(k,r)\binom{2n-2k}{n-k}=\frac{1}{2}\binom{2n+2r}{n+r}\text{,}
\end{equation}
where $n$ is a non-negative integer and $r$ is a positive integer.
It is known that Equation~(\ref{eq:12}) uniquely determines the numbers $P(n,r)$. Gessel used  Equation~(\ref{eq:12}) in order to prove Equation~(\ref{eq:1}).

Let $n$ and $r$ be positive integers.
By using Equation~(\ref{eq:12}) and Theorem \ref{t:1}, it can be proved that
\begin{equation}\label{eq:13}
\sum_{k=0}^{r-1}\binom{2k}{k}Q(n,r-k)=\frac{1}{2}\binom{2n+2r}{n+r}-\frac{1}{2}\binom{2n}{n}\binom{2r}{r}\text{.}
\end{equation}
Note that, for positive integers $n$, Equation~(\ref{eq:13}) uniquely determines the numbers $Q(n,r)$. 

\end{remark}

\section*{Acknowledgments}
I want to thank professor Du\v{s}ko Bogdani\'{c} for valuable comments which helped to improve the article.

\end{document}